\documentclass[iccmp]{ipbook}
\usepackage{enumerate}
\usepackage{amsrefs}
\usepackage{amsfonts, amsmath, amssymb, amscd, amsthm, bm, cancel}
\usepackage{url}
\usepackage{graphicx}
\usepackage[linktocpage=true,colorlinks,citecolor=magenta,linkcolor=blue,urlcolor=magenta]{hyperref}
\usepackage{multicol}
\usepackage{comment}

\startlocaldefs

\theoremstyle{plain} \newtheorem{thm}{Theorem}[section]
\theoremstyle{plain} \newtheorem{prop}[thm]{Proposition}
\theoremstyle{plain} 
\theoremstyle{definition} 
\theoremstyle{plain} \newtheorem{cor}[thm]{Corollary}
\theoremstyle{plain} \newtheorem{lem}[thm]{Lemma}
\theoremstyle{plain} 
\theoremstyle{remark} \newtheorem{rmk}[thm]{Remark}
\theoremstyle{plain} \newtheorem{conj}{Conjecture}

\newtheorem*{ack}{Acknowledgments}

\numberwithin{equation}{section}
\allowdisplaybreaks

\newcommand{\divg}{\mathrm{div}}

\newcommand{\la}{\langle}
\newcommand{\ra}{\rangle}

\endlocaldefs

\firstpage{1}
\lastpage{1}

\title[Weinstock inequality in hyperbolic space]{Weinstock inequality in hyperbolic space II}

\author{Pingxin Gu}
\address{Department of Mathematical Sciences, Tsinghua University, Beijing 100084, P.R. China}
\email{\href{gpx21@mails.tsinghua.edu.cn}{gpx21@mails.tsinghua.edu.cn}}

\author{Haizhong Li}
\address{Department of Mathematical Sciences, Tsinghua University, Beijing 100084, P.R. China}
\email{\href{lihz@tsinghua.edu.cn}{lihz@tsinghua.edu.cn}}

\author{Yao Wan}
\address{Department of Mathematics, The Chinese University of Hong Kong, Shatin, N.T., Hong Kong}
\email{\href{mailto:yaowan@cuhk.edu.hk}{yaowan@cuhk.edu.hk}}

\begin{document}

\begin{abstract}
In this paper, we establish the Weinstock inequality for the first non-zero Steklov eigenvalue on star-shaped mean convex domains in hyperbolic space $\mathbb{H}^n$ for $n\geq 3$. We note that when $n\geq 4$, the result was obtained in our previous paper \cite{GLW25}. In particular, when the domain is convex, our result gives an affirmative answer to Open Question 4.27 in \cite{CGGS24} for the hyperbolic case. 
\end{abstract}

\maketitle


\section{Introduction}\label{sec-1}

One of the oldest topics in spectral geometry is shape optimization. We focus on the first non-zero eigenvalue of the Steklov eigenvalue problem, which was first introduced by V. A. Steklov \cite{Ste02} in 1902 during his study of heat conduction. Unlike the classical Dirichlet or Neumann boundary value problems, the Steklov eigenvalue parameter appears in the boundary condition, reflecting not only the interior geometry of the manifold but also the geometric properties of the boundary and their interactions. As a natural extension of classical geometric isoperimetric inequalities \cites{Sch43,Sch39,Oss78,Ber05} to spectral theory, the shape optimization problem for Steklov eigenvalues has gained increasing prominence in geometric analysis. We refer readers to the notable surveys \cites{GP17,CGGS24}. 

Let $\Omega$ be a compact Riemannian manifold with Lipschitz boundary $\partial\Omega$. The Steklov eigenvalue problem on $\Omega$ is given by
\begin{align*}
    \left\{\begin{array}{ll}
    \Delta u=0,&\text{in $\Omega$},\\
    \frac{\partial u}{\partial \nu}=\sigma(\Omega)u,&\text{on $\partial\Omega$},
\end{array}\right.
\end{align*}
where $\Delta$ is the Laplacian operator on $\Omega$, and $\nu$ is the outward unit normal of $\partial\Omega$. We may compute the first non-zero Steklov eigenvalue of $\Omega$ by
\begin{align}\label{def-steklov}
\sigma_1(\Omega)=\min\left\{\frac{\int_{\Omega}|\nabla u|^2 dv}{\int_{\partial\Omega}u^2 d\mu}:\ u\in H^1(\Omega)\backslash \{0\},\int_{\partial\Omega}u d\mu=0\right\},
\end{align}
where $dv$ is the volume element of $\Omega$ and $d\mu$ is the area element of $\partial\Omega$. From an analytical perspective, the Steklov eigenvalues are exactly the eigenvalues of the Dirichlet-to-Neumann (DtN) operator. Specifically, for any given boundary data $f$ on $\partial\Omega$, the DtN operator $\mathcal{D}$ is defined by $\mathcal{D}f = \frac{\partial \hat{u}}{\partial \nu}$, where $\hat{u}$ is the unique harmonic extension of $f$ to the interior $\Omega$ (i.e., $\Delta \hat{u} = 0$ in $\Omega$ and $\hat{u} = f$ on $\partial\Omega$). In this way, the Steklov eigenvalue problem can be equivalently expressed as the boundary eigenvalue equation $\mathcal{D}u = \sigma(\Omega)u$. Since the DtN operator is a first-order elliptic pseudo-differential operator, its principal symbol is largely determined by the local geometry of the boundary \cites{LU89,HL01,GP14,GKL21}, which implies a direct connection between Steklov eigenvalues and boundary geometric quantities.

The earliest isoperimetric type inequality on $\sigma_1(\Omega)$ is the well-known Weinstock's inequality, which states that among simply-connected planar domains $\Omega$ with a fixed perimeter $|\partial\Omega|$, $\sigma_1(\Omega)$ is maximized by a disk; see \cite{Wei54}. Using the existence theorem of proper holomorphic mappings by Ahlfors \cite{Ahl50} and Gabard \cite{Gab06}, Weinstock's result has been widely extended to Riemann surfaces with higher genus and multiple boundary components, see \cites{HP68,HPS75,GP10,FS11,GP12,Kar17,YY17,CEG11,CEG19,Has11,FS16,Kok14,GKL21}. It is also possible to maximize $\sigma_1(\Omega)$ by fixing the volume $|\Omega|$ instead of the perimeter for simply connected domains; see \cite{Esc99}. For higher dimensions, Brock \cite{Bro01} generalized this by proving that a ball minimizes the sum of the reciprocals of the first $n$ Steklov eigenvalues, where topological and connectedness assumptions can indeed be removed, with related stability results established in \cite{BDR12}.

However, the simply-connected assumptions in Weinstock inequality cannot be removed, since the result fails for appropriate spherical shells $\Omega_{\epsilon}=B(1)\backslash \overline{B(\epsilon)}$ (where $B(r)$ denotes the geodesic ball of radius $r$ centered at the origin); see \cite{GP17}. In higher dimensions $n\geq 3$, purely topological constraints are insufficient for spectral rigidity. Fraser and Schoen \cites{FS19,FS20} introduced surgery techniques to construct smooth contractible domains in $\mathbb{R}^n$ whose perimeter-normalized first Steklov eigenvalues are strictly larger than that of a Euclidean ball. Hong \cite{Hon21} further extended this to any finite number of Steklov eigenvalues. 

Nevertheless, when geometric constraints are introduced, rigidity can be recovered. Using the inverse mean curvature flow (IMCF), Bucur, Ferone, Nitsch and Trombetti \cite{BFNT21} proved that

\begin{thm}[\cite{BFNT21}]\label{Thm-BFNT-3.1}
Let $\Omega$ be a bounded convex domain in $\mathbb{R}^n$. Then
\begin{align}\label{eq-thmbfnt}
    \sigma_1(\Omega)|\partial\Omega|^{\frac{1}{n-1}}\leq \sigma_1(B)|\partial B|^{\frac{1}{n-1}},
\end{align}
where $B\subset \mathbb{R}^n$ is a ball, and equality holds if and only if $\Omega$ is a ball.
\end{thm}

Combining IMCF with elementary methods, Kwong and Wei \cite{KW23} proved sharp geometric inequalities involving three quantities, and extended Theorem \ref{Thm-BFNT-3.1} to a broader class of star-shaped mean convex domains. 

In curved spaces, the situation exhibits diverse phenomena. In positively curved spaces, Euclidean spectral rigidity often fails. For instance, in the unit sphere $\mathbb{S}^n$, Castillon and Ruffini \cite{CR19} constructed a ring-shaped domain whose first Steklov eigenvalue is strictly larger than that of a geodesic ball of the same volume. Conversely, in negatively curved spaces, rigidity tends to be stronger. For the volume-constrained problem, Binoy and Santhanam \cite{BS14} proved that the geodesic ball is the unique maximizer for $\sigma_1(\Omega)$ in the hyperbolic space $\mathbb{H}^n$. Therefore, it is natural to ask whether a perimeter-constrained Weinstock type inequality holds in hyperbolic space, specifically:

\begin{conj}[\cite{CGGS24}, Open Question 4.27]\label{conj}
Let $\Omega$ be a bounded convex domain in $\mathbb{H}^n$. Let $\Omega^*$ be a geodesic ball with $|\partial\Omega^*|=|\partial\Omega|$. Is it true that
\begin{equation}
    \sigma_1(\Omega)\leq\sigma_1(\Omega^*)
\end{equation}
and equality holds if and only if $\Omega$ is a geodesic ball?
\end{conj}

For $n=2$, the conjecture holds, as the Weinstock inequality is valid for the hyperbolic plane and, more generally, for compact surfaces, which was remarkably proved by Fraser and Schoen \cite{FS11} using conformal mappings. For higher dimensions, the volume-constrained version was solved (see \cite{BS14}). When $\Omega$ is convex, this can be viewed as a corollary of Conjecture \ref{conj} by means of the isoperimetric inequality. Recently, the perimeter-constrained Conjecture \ref{conj} for $n \ge 4$ was established by the authors in \cite{GLW25}. 

In this paper, we completely settle this problem. By developing a weighted mass transplantation method to bridge the gap for $n=3$, we prove the Weinstock inequality for smooth domains with star-shaped mean convex boundary in $\mathbb{H}^n$ for general dimension $n\geq 3$, thereby providing a complete answer to the conjecture.

\begin{thm}\label{main thm}
Let $n\geq 3$, $\Omega$ be a smooth bounded domain in $\mathbb{H}^n$ with star-shaped mean convex boundary $\partial\Omega$, and $\Omega^{\ast}$ be a geodesic ball with $|\partial\Omega^*|=|\partial\Omega|$. Then
\begin{equation}\label{main-equ}
    \sigma_1(\Omega)\leq \sigma_1(\Omega^{\ast})
\end{equation}
and equality holds if and only if $\Omega$ is a geodesic ball.
\end{thm}

\begin{rmk}
The result in Theorem \ref{main thm} also holds for star-shaped and weakly mean convex domains in hyperbolic space via an approximation argument.
\end{rmk}

In particular, when the domain is convex, we provide an affirmative answer to Conjecture \ref{conj}.

\begin{cor}
     Conjecture \ref{conj} is true.
\end{cor}

\begin{rmk}\label{rmk-1}
Let $I_1(t)=|\partial B(t)|$ and $J(t)=\sigma_1(B(t))$ for $t>0$. Then (\ref{main-equ}) is equivalent to the following isoperimetric-type inequality
\begin{equation}\label{main-equ2}
    \frac{\sigma_1(\Omega)}{J\circ I_1^{-1}(|\partial\Omega|)}\leq \left.\frac{\sigma_1(\Omega)}{J\circ I_1^{-1}(|\partial\Omega|)}\right|_{\text{$\Omega$ is a geodesic ball}}=1.
\end{equation}
\end{rmk}

Our method is inspired by the use of IMCF as presented in \cite{KW23} and by the well-known mass transplantation argument, as discussed in \cite{FL21}. The application of IMCF is not straightforward, as we must introduce a special auxiliary function $h$ in \eqref{def-h}. The evolution of this function along IMCF establishes a useful geometric inequality. By the same estimate, as a byproduct, we can extend Verma's isoperimetric inequality \cite{Ver21} for the harmonic mean of the first $(n-1)$ non-zero Steklov eigenvalues with prescribed area. We state it as the following Corollary \ref{main cor}, which can lead to Theorem \ref{main thm} again.

\begin{cor}\label{main cor}
Let $n\geq 3$, $\Omega$ be a smooth bounded domain in $\mathbb{H}^n$ with star-shaped mean convex boundary $\partial\Omega$, and $\Omega^{\ast}$ be a geodesic ball with $|\partial\Omega^*|=|\partial\Omega|$. Then
\begin{equation}
    \sum_{i=1}^{n-1}\frac{1}{\sigma_i(\Omega)}\geq \sum_{i=1}^{n-1}\frac{1}{\sigma_i(\Omega^{\ast})}
\end{equation}
and equality holds if and only if $\Omega$ is a geodesic ball.
\end{cor}

$\ $

The paper is organized as follows. In Section \ref{sec-2}, we present basic properties of hyperbolic spaces, including the isoperimetric inequality and the first non-zero Steklov eigenvalue of geodesic balls. We also state the well-known theorems for center of mass and mass transplantation. In Section \ref{sec-3}, we discuss the detailed analytic properties of the volume-area ratio function $g$. In Section \ref{sec-4}, we introduce the key auxiliary function $h$, obtain the estimate using inverse mean curvature flow, and prove Theorem \ref{main thm} and Corollary \ref{main cor}. Specifically, we apply the mass transplantation argument for higher dimensions ($n\geq 5$), and develop a weighted mass transplantation method for lower dimensions ($n=3,4$) to overcome the analytical difficulties caused by the lack of global monotonicity.

\begin{ack}
	This work was supported by National Natural Science Foundation of China [grant number 12471047]. 
    The third author was supported by Hong Kong RGC grant (Early Career Scheme) of Hong Kong No. 24304222 and No. 14300623, and NSFC grant No. 12222122.
    We would like to thank Yong Wei for his valuable comments and suggestions.
\end{ack}

\section{Preliminaries}\label{sec-2}
We use the warped product to illustrate the hyperbolic space $(\mathbb{H}^n,g_{\mathbb{H}^n})=[0,\infty)\times_{\lambda(r)} \mathbb{S}^{n-1}$, where $\lambda(r)=\sinh r$. That is, after choosing an origin $O$ in $\mathbb{H}^n$, the Riemannian metric $g_{\mathbb{H}^n}$ on $\mathbb{H}^n$ is defined by
\[
g_{\mathbb{H}^n}=dr^2+\lambda^2(r)g_{\mathbb{S}^{n-1}},
\]
where $r$ denotes the geodesic distance from $O$, and $g_{\mathbb{S}^{n-1}}$ is the canonical metric on the unit sphere $\mathbb{S}^{n-1}$. We will also use the notation $\la \cdot,\cdot\ra$ to denote $g_{\mathbb{H}^n}$. 

Let $\Lambda(r) = \int_0^r \lambda(\rho) d\rho = \cosh r - 1$. The hyperbolic space $\mathbb{H}^n$ admits a conformal Killing vector field $V = \bar{\nabla}\Lambda = \lambda(r)\partial_r$. This vector field satisfies the following differential equation (see \cite{Bre13}):
\begin{equation}
\bar{\nabla} V = \lambda'(r)g_{\mathbb{H}^n},
\end{equation}
where $\bar{\nabla}$ is the Levi-Civita connection of $\mathbb{H}^n$.

\subsection{The first non-zero Steklov eigenvalue of geodesic balls}$\ $

For $r> 0$, a geodesic ball $B(r)\subset \mathbb{H}^n$ centered at $O$ with radius $r$ is given by $\{(s,\theta)\in \mathbb{H}^n:\ 0\leq s\leq r\}\subset \mathbb{H}^n$, and its boundary is the geodesic sphere $S(r)=\partial B(r)=\{(s,\theta)\in \mathbb{H}^n:\ s=r\}$. We can directly compute the volume and area of $B(r)$ as
\begin{align}\label{eq-B}
|B(r)|=\int_0^r\int_{\mathbb{S}^{n-1}}\lambda^{n-1}(t)dtd\mu_{\mathbb{S}^{n-1}}=\omega_{n-1}\int_0^r\lambda^{n-1}(t)dt
\end{align}
and
\begin{align}\label{eq-S}
|S(r)|=\int_{\mathbb{S}^{n-1}}\lambda^{n-1}(r)d\mu_{\mathbb{S}^{n-1}}=\omega_{n-1}\lambda^{n-1}(r),
\end{align}
where $\omega_{n-1}$ denotes the area of $\mathbb{S}^{n-1}$. 

It is remarkable that the geodesic ball minimizes the area among bounded domains with the same volume,  which is known as the isoperimetric inequality in hyperbolic space.
\begin{thm}[\cite{Sch43}]\label{thm-isoperimetric inequality}
Let $\Omega\subset \mathbb{H}^n$ be a bounded domain and $B(R)$ be a geodesic ball with the same volume as $\Omega$. Then the area of $\Omega$ satisfies
\begin{align}\label{eq-isoperimetric}
|\partial\Omega|\geq |\partial B(R)|.
\end{align}
Furthermore, equality holds if and only if $\Omega$ is a geodesic ball.
\end{thm}

\begin{rmk}\label{rmk-2}
As noted in Remark \ref{rmk-1}, if we additionally denote $I_0(t)=|B(t)|$ for $t>0$, then \eqref{eq-isoperimetric} is equivalent to the following inequality
\begin{align}
\frac{|\Omega|}{I_0\circ I_1^{-1}(|\partial\Omega|)}\leq \left.\frac{|\Omega|}{I_0\circ I_1^{-1}(|\partial\Omega|)}\right|_{\text{$\Omega$ is a geodesic ball}}=1.
\end{align}
Combining this with (\ref{main-equ2}), we deduce
\begin{align}
    \frac{\sigma_1(\Omega)}{J\circ I_0^{-1}(|\Omega|)}\leq \left.\frac{\sigma_1(\Omega)}{J\circ I_0^{-1}(|\Omega|)}\right|_{\text{$\Omega$ is a geodesic ball}}=1.
\end{align}
\end{rmk}

An important function in hyperbolic space is the geodesic ball's volume-area ratio function $g$, defined as $g(r):=\frac{|B(r)|}{|S(r)|}$, i.e.,
\begin{align}\label{eq-def-g}
g(r)=\frac{1}{\lambda^{n-1}(r)}\int_0^r\lambda^{n-1}(t)dt.
\end{align}
Using the standard method of separation of variables, the first non-zero Steklov eigenvalue of a geodesic ball is computed as follows:
\begin{prop}[\cite{BS14}]
Let $B(r)\subset \mathbb{H}^n$ be a geodesic ball centered at $O$ with boundary $S(r)$. Then the first non-zero eigenvalue $\sigma_1(B(r))$ of the Steklov eigenvalue problem \eqref{def-steklov} on $B(r)$ is given by
\begin{align}\label{eq-sigmabr}
\sigma_1(B(r))=\frac{\int_{B(r)}((g')^2+(n-1)\frac{g^2}{\lambda^2}) dv}{g(r)^2|S(r)|}.
\end{align}
\end{prop}

\subsection{Center of mass}$\ $

In order to study the upper bound of the first non-zero Steklov eigenvalue not only for geodesic balls but also for general domains, we need the following center of mass theorem. 

In \cite{AS96}, Aithal and Santhanam give a theorem for complete Riemannian manifold $M$, which also applies to measurable subset $A$ in hyperbolic space. Let $CA$ denote the convex hull of $A$. Let $\exp_q:T_qM\to M$ be the exponential map, and let $X=(x_1,\ldots,x_n)$ be a system of normal coordinates centered at $q$. We identify $CA$ with $\exp_q^{-1}(CA)$ for each $q\in CA$, and denote $(g_{\mathbb{H}^n})_q(X,X)$ as $\|X\|_q^2$ for $X\in T_qM$. The theorem of center of mass with respect to the mass distribution function $G$ is stated as
\begin{thm}[\cite{BS14}]
Let $A\subset \mathbb{H}^n$ be a measurable subset and $G$ be a continuous function on $[0,\infty)$ that is positive on $(0,\infty)$. Then there exists a point $p\in CA$ such that
\[
\int_{A}G(\|X\|_p)X dv=0,
\]
where $X=(x_1,\ldots,x_n)$ is a normal coordinate system centered at $p$.
\end{thm}

By applying the center of mass theorem for $A=\partial\Omega$ with respect to $G=g$ and choosing test functions $g\frac{x_i}{r}$ in the Rayleigh quotient \eqref{def-steklov}, we obtain
\begin{align*}
    \sigma_1(\Omega)\int_{\partial\Omega}g^2\frac{x_i^2}{r^2} d\mu\leq \int_{\Omega}|\nabla (g\frac{x_i}{r})|^2 dv,\qquad i=1,\ldots,n.
\end{align*}
Summing over $i$, we derive the upper bound of the first non-zero Steklov eigenvalue via $g$. For further details, see, e.g. \cites{AV22, BS14, Ver21}.
\begin{prop}\label{prop-upperbound-steklov}
Let $\Omega$ be a bounded domain in $\mathbb{H}^n$ with smooth boundary $\Sigma=\partial \Omega$. Then
\begin{align}\label{eq-upperbound-steklov}
\sigma_1(\Omega)\int_{\Sigma}g^2 d\mu\leq \int_{\Omega}\left((g')^2+(n-1)\frac{g^2}{\lambda^2}\right) dv.
\end{align}
\end{prop}

\subsection{Inverse mean curvature flow}\label{subsec-imcf}$\ $

Curvature flows are important tools in the study of geometric inequalities. In this subsection, we recall the basic definitions and geometric evolution properties of the inverse mean curvature flow (IMCF) in hyperbolic space.

For a smooth closed hypersurface $\Sigma$ in $\mathbb{H}^n$, let $\nu$ be its outward unit normal vector field. The second fundamental form is defined as $h(X,Y)=\la \bar\nabla_X\nu,Y\ra$, and the mean curvature $H$ is the trace of the second fundamental form, i.e., $H=\mathrm{tr}_g h$. We say $\Sigma$ is mean convex if its mean curvature is strictly positive everywhere. Furthermore, $\Sigma$ is called star-shaped if it can be represented as a radial graph over a point, which is geometrically equivalent to $\la \partial_r,\nu\ra>0$ everywhere.

Let $\Omega_0\subset \mathbb{H}^n$ be a bounded domain with a smooth, star-shaped, and mean convex boundary $\Sigma_0=\partial\Omega_0$. Gerhardt \cite{Ger11} proved the well-posedness and asymptotic behavior of such hypersurfaces evolving along IMCF:
\begin{thm}[\cite{Ger11}]
Let $\Sigma_0\subset \mathbb{H}^n$ be a smooth, star-shaped, and mean convex closed hypersurface. Then there exists a family of smooth hypersurfaces $\Sigma_t=X(\cdot,t)$ satisfying the inverse mean curvature flow equation
\begin{equation}\label{eq-IMCF}
\left\{\begin{array}{l}\frac{\partial}{\partial t}X(x,t) = \frac{1}{H(x,t)}\nu(x,t),\\
X(\cdot,0) = \Sigma_0, \end{array}\right.
\end{equation}
where $\nu(x,t)$ is the outward unit normal of $\Sigma_t$.

The flow exists for all time $t\in [0,\infty)$. For any $t>0$, the evolving hypersurface $\Sigma_t$ remains smooth, star-shaped, and mean convex. As $t\to \infty$, $\Sigma_t$ expands to infinity, and all its principal curvatures converge exponentially to $1$.
\end{thm}

During the evolution of IMCF, the evolution equations of geometric quantities are crucial for establishing monotone formulas.
\begin{prop}
Along the inverse mean curvature flow \eqref{eq-IMCF}, the induced metric $g_{ij}$ and the area element $d\mu_t$ satisfy the following evolution equations:
\begin{align}
\frac{\partial}{\partial t} g_{ij} &= \frac{2}{H} h_{ij}, \\
\frac{\partial}{\partial t} d\mu_t &= d\mu_t, \label{eq-dmut}
\end{align}
where $h_{ij}$ is the second fundamental form. In particular, \eqref{eq-dmut} implies that the surface area grows exponentially, i.e., $|\Sigma_t|=|\Sigma_0|e^t$.

Furthermore, for a smooth function $F$ defined on $\Omega_t$ that does not depend on time $t$, the coarea formula yields the following integral evolution formula:
\begin{equation}\label{eq-dvt}
\frac{d}{dt}\int_{\Omega_t} F dv = \int_{\Sigma_t} \frac{F}{H} d\mu_t.
\end{equation}
\end{prop}

\subsection{Mass transplantation}$\ $

Finally, we state the mass transplantation argument due to Weinberger \cite{Wei56}, which is a powerful tool for estimates involving integrals. A clear proof can be seen in \cite{FL21}.
\begin{thm}[Mass transplantation]\label{masstransplantation}
Let $\Omega$ be a bounded Lipschitz domain in $\mathbb{H}^n$ with the same volume as $B(R)$. If $f(r)$ is decreasing and integrable on $[0,+\infty)$, then
\begin{equation}\label{mass-ineq}
    \int_{\Omega}f(r) dv\leq\int_{B(R)}f(r) dv.
\end{equation}
If in addition $f(r)$ is strictly decreasing, then equality holds if and only if $\Omega=B(R)$. If $f(r)$ is increasing and integrable on $[0,+\infty)$, then the inequality reverses direction.
\begin{proof}
Since $f$ is decreasing and $|\Omega|=|B(R)|$, we have
\begin{equation}\label{mass-ineq2}
\begin{split}
     \int_{\Omega}f(r) dv
    =&\int_{\Omega\cap B(R)}f(r) dv+\int_{\Omega\backslash B(R)}f(r) dv\\
    \leq& \int_{\Omega\cap B(R)}f(r) dv+|\Omega\backslash B(R)|f(R)\\
    =&\int_{\Omega\cap B(R)}f(r) dv+|B(R)\backslash \Omega|f(R)\\
    \leq& \int_{\Omega\cap B(R)}f(r) dv+\int_{B(R)\backslash \Omega}f(r) dv=\int_{B(R)}f(r) dv.
\end{split}
\end{equation}
It suffices to show that the inequality (\ref{mass-ineq}) is strict when $f(r)$ is strictly decreasing. If $\Omega\nsubseteq B(R)$, then $\Omega$ contains a point at radius $r>R$ and thus contains a neighborhood outside $B(R)$. Consequently, $|\Omega\backslash B(R)|>0$, and the first inequality in (\ref{mass-ineq2}) is strict because $f(r)$ is strictly decreasing. Similarly, if $B(R)\nsubseteq \Omega$, the second inequality in (\ref{mass-ineq2}) is strict. This completes the proof of Theorem \ref{masstransplantation}.
\end{proof}
\end{thm}

\section{The volume-area ratio function $g(r)$}\label{sec-3}
In this section, we study the properties of $g(r)$ defined in \eqref{eq-def-g}.

\subsection{The geometric properties of $g(r)$}$\ $

Notice that the radial part $g(r)$ of the first non-zero Steklov eigenfunction of a geodesic ball is geometrically exactly the ratio of the volume to the surface area of the geodesic ball $B(r)$. This geometric intuition can be rigorously characterized by the following divergence property.

By using the expression for $g(r)$ given in \eqref{eq-def-g}, we have
\begin{align}\label{eq-deri-ln-1g}
(\lambda^{n-1}(r)g(r))'=\left(\int_0^r\lambda^{n-1}(t)dt\right)'=\lambda^{n-1}(r),
\end{align}
and then
\begin{align}\label{eq-deri-g}
g'=1-(n-1)\frac{\lambda' g}{\lambda}.
\end{align}

\begin{prop}\label{prop-divgY}
Let $Y:= g(r) \partial_r$ be a vector field on $\mathbb{H}^n$. Then its divergence is identically $1$, i.e.,
\begin{equation}\label{div-Y}
    \divg(Y)=1.
\end{equation}

\begin{proof}
For each $p\in \mathbb{H}^n$, choose an orthonormal frame $\{e_1,\ldots,e_n\}$ around $p$ such that $e_n=\partial_r$. Recall the vector field $V:=\lambda(r)\partial_r$ is a conformal Killing field on $\mathbb{H}^n$, which satisfies
\begin{align}\label{eq-killing}
\la \bar\nabla_{e_i}V,e_j\ra=\lambda'(r)\la e_i,e_j\ra,
\end{align}
see e.g. \cite[Lemma 2.2]{Bre13}. We can write $Y$ as $Y = \frac{g}{\lambda}V$. Then it follows from \eqref{eq-deri-g} that
\begin{align*}
    \divg(Y)=&\divg\left(\frac{g}{\lambda}V\right)
    =\la \bar\nabla\left(\frac{g}{\lambda}\right), V\ra + \frac{g}{\lambda}\divg(V)\\
    =&\left(\frac{g}{\lambda}\right)'\lambda+n\left(\frac{g}{\lambda}\right)\lambda'
    =g'+(n-1)\frac{\lambda' g}{\lambda}=1.
\end{align*}
\end{proof}
\end{prop}

A direct corollary follows from Proposition \ref{prop-divgY} and the divergence theorem, which allows us to control the volume of a bounded domain using the boundary integral of $g(r)$.

\begin{cor}\label{cor-g}
For any bounded domain $\Omega \subset \mathbb{H}^n$, we have the following inequality:
\begin{align}\label{eq-g}
\int_{\partial\Omega}g d\mu\geq |\Omega|.
\end{align}
Equality holds if and only if $\Omega$ is a geodesic ball centered at the origin.
\begin{proof}
Applying the divergence theorem on $\Omega$ and using Proposition \ref{prop-divgY}, since $\la \partial_r,\nu\ra\leq 1$ and $g\geq 0$, we have
\begin{align*}
    |\Omega|=\int_{\Omega}dv=\int_{\Omega}\divg(g(r)\partial_ r) dv=\int_{\partial\Omega}g\la \partial_r,\nu\ra d\mu\leq \int_{\partial\Omega}g d\mu.
\end{align*}
Equality holds if and only if $\nu = \partial_r$ everywhere on $\partial\Omega$, which implies that $r$ is constant on the boundary, meaning $\Omega$ must be a geodesic ball centered at the origin.
\end{proof}
\end{cor}

Similarly, analogous to the volume control, we can derive a lower bound for the boundary integral of $g^2$ using the mass transplantation argument.

\begin{prop}\label{prop-g2}
For any bounded domain $\Omega \subset \mathbb{H}^n$, let $B(R)$ be a geodesic ball with the same volume as $\Omega$, i.e., $|B(R)| = |\Omega|$. Then
\begin{align}\label{eq-intg2}
\int_{\partial\Omega}g^2 d\mu\geq g(R)|B(R)|.
\end{align}
Equality holds if and only if $\Omega$ is a geodesic ball.
\begin{proof}
First, it follows from \eqref{eq-deri-g} and \eqref{eq-killing} that
\begin{align*}
    \divg(g^2\partial_r)=\divg\left(\frac{g^2}{\lambda}V\right)=\left(\frac{g^2}{\lambda}\right)'\lambda+n\left(\frac{g^2}{\lambda}\right)\lambda'=g(1+g'),
\end{align*}
and then
\begin{align*}
    \int_{\partial\Omega}g^2 d\mu\geq \int_{\partial\Omega}g^2\la \partial_r,\nu\ra d\mu=\int_{\Omega}g(1+g') dv.
\end{align*}

Next, we demonstrate the monotonicity of $g(1+g')$ as follows:
\begin{align*}
    (g(1+g'))'=& g'+(g')^2+gg''\\
    =& g'+(g')^2+g\left((n-1)\frac{g}{\lambda^2}-(n-1)\frac{\lambda'g'}{\lambda}\right)\\
    =& g'\left(1-(n-1)\frac{\lambda'g}{\lambda}\right)+(g')^2+(n-1)\frac{g^2}{\lambda^2}\\
    =& 2(g')^2+(n-1)\frac{g^2}{\lambda^2}\geq 0.
\end{align*}
Since $g(1+g')$ is a radially increasing function, using the mass transplantation argument, we conclude
\begin{align*}
    \int_{\Omega}g(1+g') dv & \geq \int_{B(R)}g(1+g') dv
    = \omega_{n-1}\int_0^R g(1+g')\lambda^{n-1}dt\\
    &= \omega_{n-1}\int_0^R (\lambda^{n-1}g^2)' dt
    = \omega_{n-1} \lambda^{n-1}(R)g^2(R)\\
    &= g(R) |B(R)|.
\end{align*}
This completes the proof of (\ref{eq-intg2}).
\end{proof}
\end{prop}

\subsection{The analytic properties of $g(r)$}$\ $

To construct effective monotone quantities, we need to thoroughly examine the differential properties of $g(r)$. The following properties show that $g(r)$ is an increasing concave function with specific asymptotic behaviors.

\begin{prop}\label{prop-g}
The function $g$ is an increasing concave function, with the following limits:
\begin{align}\label{eq-g-lim}
\lim_{r\to 0}\frac{g}{\lambda}=\frac{1}{n},\qquad \lim_{r\to \infty}g=\frac{1}{n-1},\qquad \lim_{r\to 0}g'=\frac{1}{n},\qquad \lim_{r\to \infty}g'=0.
\end{align}
Furthermore, if $n\geq 4$, we have
\begin{align}\label{eq-g-lim2}
\lim_{r\to\infty}\lambda^2g'=\frac{1}{n-3}.
\end{align}
\begin{proof}
First, since $\left.\lambda^{n-1}\left(g-\frac{1}{n-1}\frac{\lambda}{\lambda'}\right)\right|_{r=0}=0$ and
\begin{align*}
    \left(\lambda^{n-1}(g-\frac{1}{n-1}\frac{\lambda}{\lambda'})\right)'=\lambda^{n-1}-\frac{1}{n-1}\frac{n\lambda^{n-1}(\lambda')^2-\lambda^{n+1}}{(\lambda')^2}\leq 0,
\end{align*}
we deduce the upper bound estimate
\begin{align*}
    g\leq \frac{1}{n-1}\frac{\lambda}{\lambda'},
\end{align*}
and substituting this into the expression for $g'$ gives
\begin{align*}
    g'=1-(n-1)\frac{\lambda' g}{\lambda}\geq 0.
\end{align*}

Furthermore, since $\left.\lambda^{n-1}\left(g-\frac{\lambda\lambda'}{(n-1)\lambda^2+n}\right)\right|_{r=0}=0$ and
\begin{align*}
    \left(\lambda^{n-1}\left(g-\frac{\lambda\lambda'}{(n-1)\lambda^2+n}\right)\right)' 
    &= \lambda^{n-1} \left( 1 - \frac{((n-1)\lambda^2+n)^2 + 2\lambda^2}{((n-1)\lambda^2+n)^2} \right) \\
    &= -\frac{2\lambda^{n+1}}{((n-1)\lambda^2+n)^2} \leq 0,
\end{align*}

we have a finer upper bound
\begin{align*}\label{eq-gll'}
g\leq \frac{\lambda\lambda'}{(n-1)\lambda^2+n},
\end{align*}
which implies
\begin{align*}
    g''=(n-1)\frac{g}{\lambda^2}-(n-1)g'\frac{\lambda'}{\lambda}=\frac{n-1}{\lambda^2}(g((n-1)\lambda^2+n)-\lambda\lambda')\leq 0.
\end{align*}
This proves that $g(r)$ is an increasing concave function.

Finally, the limits follow directly from l'H\^opital's rule:
\begin{align*}
    \lim_{r\to 0}\frac{g}{\lambda}
    =& \lim_{r\to 0}\frac{\int_0^r\lambda^{n-1}(t)dt}{\lambda^n(r)}=\lim_{r\to 0}\frac{\lambda^{n-1}(r)}{n\lambda^{n-1}(r)\lambda'(r)}=\frac{1}{n},\\
    \lim_{r\to \infty}g
    =& \lim_{r\to \infty}\frac{\int_0^r\lambda^{n-1}(t)dt}{\lambda^{n-1}(r)}=\lim_{r\to \infty}\frac{\lambda^{n-1}(r)}{(n-1)\lambda^{n-2}(r)\lambda'(r)}=\frac{1}{n-1},\\
    \lim_{r\to 0}g'
    =& 1-(n-1)\lim_{r\to 0}\frac{\lambda'g}{\lambda}=1-(n-1)\lim_{r\to 0}\frac{\int_0^r\lambda^{n-1}(t)dt}{\lambda^n(r)}=1-\frac{n-1}{n}=\frac{1}{n},\\
    \lim_{r\to \infty}g'
    =& 1-(n-1)\lim_{r\to \infty}\frac{\lambda'g}{\lambda}=1-(n-1)\cdot\frac{1}{n-1}=0,
\end{align*}
and if $n\geq 4$,
\begin{align*}
    \lim_{r\to \infty}\lambda^2g'
    =& \lim_{r\to \infty}\frac{\lambda^n-(n-1)\lambda'\int_0^r\lambda^{n-1}}{\lambda^{n-2}}\\
    =& \lim_{r\to\infty}\frac{\frac{\lambda^n}{\lambda'}-(n-1)\int_0^r\lambda^{n-1}}{\lambda^{n-3}}\\
    =& \lim_{r\to\infty}\frac{\frac{(n-1)\lambda^{n+1}+n\lambda^{n-1}}{\lambda^2+1}-(n-1)\lambda^{n-1}}{(n-3)\lambda^{n-4}\lambda'}\\
    =& \lim_{r\to\infty}\frac{\lambda^{n-1}}{(n-3)\lambda^{n-4}\lambda'(\lambda^2+1)}=\frac{1}{n-3}.
\end{align*}

\end{proof}
\end{prop}

Since $\frac{\lambda' g}{\lambda}=\frac{1}{n-1}(1-g')$, Proposition \ref{prop-g} immediately yields the following proposition.
\begin{prop}\label{prop-l'g/l}
The function
\begin{align*}
    \frac{\lambda'(r)g(r)}{\lambda(r)}
\end{align*}
is increasing on $(0,\infty)$, and satisfies the bounds
\begin{align*}
    \frac{1}{n}\leq \frac{\lambda'(r)g(r)}{\lambda(r)}\leq \frac{1}{n-1}.
\end{align*}
\end{prop}

On the other hand, we also need to control higher-order terms involving derivatives:
\begin{prop}\label{prop-ll'g'/g}
The function
\begin{align*}
    \frac{\lambda(r)\lambda'(r)g'(r)}{g(r)}
\end{align*}
is increasing on $(0,\infty)$. In particular, if $n\geq 4$, we have the upper bound estimate:
\begin{align*}
    \frac{\lambda(r)\lambda'(r)g'(r)}{g(r)}\leq \frac{n-1}{n-3}.
\end{align*}

\begin{proof}
First, we compute the derivative to show that $\frac{\lambda\lambda'g'}{g}$ is increasing:
\begin{align*}
    \left(\frac{\lambda\lambda'g'}{g}\right)'=\left(\frac{\lambda\lambda'}{g}-(n-1)(\lambda^2+1)\right)'=\frac{((n+1)\lambda^2+n)g-\lambda\lambda'}{g^2}-2(n-1)\lambda\lambda'.
\end{align*}
It suffices to show that the numerator is non-negative, which is equivalent to verifying that $g(r)$ satisfies the following quadratic inequality:
\begin{align*}
    2(n-1)\lambda\lambda'g^2 - ((n+1)\lambda^2+n)g + \lambda\lambda' \leq 0,
\end{align*}
which can be rewritten as
\[
\left(g-\frac{(n+1)\lambda^2+n}{4(n-1)\lambda\lambda'}\right)^2\leq \frac{(n-3)^2\lambda^4+2(n^2-3n+4)\lambda^2+n^2}{(4(n-1)\lambda\lambda')^2}.
\]

For $n=2$, direct calculation shows that $\frac{\lambda\lambda'g'}{g}=\lambda'$ is increasing. From now on, we focus on the case $n\geq 3$. By Proposition \ref{prop-g}, we have
\begin{align*}
    g\leq \frac{\lambda}{(n-1)\lambda'}\leq \frac{(n+1)\lambda^2+n}{4(n-1)\lambda\lambda'}.
\end{align*}
This implies $g(r)$ is always smaller than the axis of symmetry of the quadratic function, thus we need only to prove that $g(r)$ is greater than the smaller root:
\begin{align*}
    g\geq \frac{(n+1)\lambda^2+n-\sqrt{(n-3)^2\lambda^4+2(n^2-3n+4)\lambda^2+n^2}}{4(n-1)\lambda\lambda'}.
\end{align*}

For convenience, denote
\begin{align*}
    A=\sqrt{(n-3)^2\lambda^4+2(n^2-3n+4)\lambda^2+n^2},
\end{align*}
then $A'=\frac{2(n-3)^2\lambda^3\lambda'+2(n^2-3n+4)\lambda\lambda'}{A}$. It suffices to show that
\begin{align*}
    \lambda^{n-1}g\geq \frac{2\lambda^n\lambda'}{(n+1)\lambda^2+n+A}.
\end{align*}
Equality holds at $r=0$. By taking a derivative, we need only to prove that 
{\small\begin{align*}
    & \lambda^{n-1}\geq \frac{2((n+1)\lambda^{n+1}+n\lambda^{n-1})((n+1)\lambda^2+n+A)-2\lambda^n\lambda'(2(n+1)\lambda\lambda'+A')}{((n+1)\lambda^2+n+A)^2} \\
    \Longleftrightarrow & ((n+1)\lambda^2+n+A)^2\geq 2((n+1)\lambda^2+n)((n+1)\lambda^2+n+A)-2\lambda\lambda'(2(n+1)\lambda\lambda'+A') \\
    \Longleftrightarrow & 2\lambda\lambda'\left(2(n+1)\lambda\lambda'+\frac{2(n-3)^2\lambda^3\lambda'+2(n^2-3n+4)\lambda\lambda'}{A}\right)\geq ((n+1)\lambda^2+n)^2-A^2\\
    \Longleftrightarrow & 2\lambda\left(2(n+1)\lambda+\frac{2(n-3)^2\lambda^3+2(n^2-3n+4)\lambda}{A}\right)(\lambda^2+1)\geq 8(n-1)\lambda^2(\lambda^2+1)\\
     \Longleftrightarrow & (n-3)^2\lambda^2+(n^2-3n+4)\geq (n-3)A.
\end{align*}}

By squaring both sides of the above inequality, it suffices to show that
\begin{align*}
         & (n-3)^4\lambda^4+2(n-3)^2(n^2-3n+4)\lambda^2+(n^2-3n+4)^2 \\
    \geq &  (n-3)^2((n-3)^2\lambda^4+2(n^2-3n+4)\lambda^2+n^2),
\end{align*}
which is clearly correct. Therefore, the function $\frac{\lambda\lambda'g'}{g}$ is increasing. 

Moreover, when $n \geq 4$, it follows from \eqref{eq-g-lim} and \eqref{eq-g-lim2} that
\[
\frac{\lambda\lambda'g'}{g}\leq \lim_{r\to\infty}\frac{\lambda\lambda'g'}{g}=\lim_{r\to\infty}\frac{\lambda^2g'}{g}=\frac{n-1}{n-3}.
\]
\end{proof}
\end{prop}

By Proposition \ref{prop-ll'g'/g}, the function $\frac{\lambda^2g'}{g}=\frac{\lambda\lambda'g'}{g}\cdot\frac{\lambda}{\lambda'}$ is a product of two positive increasing functions. Hence, we have
\begin{prop}\label{prop-l2g'/g-n=4}
The function
\begin{align*}
    \frac{\lambda^2(r) g'(r)}{g(r)}
\end{align*}
is increasing.
\end{prop}

Notice that 
\begin{align}
    \divg(gg'\partial_r)=(g')^2+(n-1)\frac{g^2}{\lambda^2}.
\end{align}
Then the expression \eqref{eq-sigmabr} for $\sigma_1(B(r))$ can be rewritten as
\begin{equation}\label{eq-sigmabr2}
    \sigma_1(B(r))=\frac{\int_{S(r)}gg'd\mu}{g(r)^2|S(r)|}=\frac{g'(r)}{g(r)}.
\end{equation}
Combining Proposition \ref{prop-l2g'/g-n=4} with (\ref{eq-S}) and (\ref{eq-sigmabr2}), we obtain the following key lemma.

\begin{lem}\label{lem-ln-1g'/g}
In the hyperbolic space $\mathbb{H}^n$ ($n \ge 3$), the geometric quantities
\begin{align}\label{eq-mono-n=4}
|S(r)|^{\frac{2}{n-1}}\sigma_1(B(r))
\end{align}
and
\begin{align}\label{eq-mono-n>=5}
|S(r)|\sigma_1(B(r))
\end{align}
are both increasing functions of the radius $r$.
\begin{proof}
Using $|S(r)| = \omega_{n-1}\lambda(r)^{n-1}$ and $\sigma_1(B(r)) = \frac{g'(r)}{g(r)}$, we have
\begin{equation*}
|S(r)|^{\frac{2}{n-1}}\sigma_1(B(r)) = (\omega_{n-1})^{\frac{2}{n-1}} \lambda^2(r) \frac{g'(r)}{g(r)}.
\end{equation*}
According to Proposition \ref{prop-l2g'/g-n=4}, $\frac{\lambda^2 g'}{g}$ is increasing, so the first geometric quantity is increasing. 

For the second geometric quantity, we can write it as:
\begin{equation*}
|S(r)|\sigma_1(B(r)) = \omega_{n-1} \lambda(r)^{n-3} \cdot \left( \lambda^2(r) \frac{g'(r)}{g(r)} \right).
\end{equation*}
When $n \geq 3$, $\lambda(r)^{n-3}$ is non-decreasing. Since the product of two positive increasing functions is still an increasing function, the conclusion holds.
\end{proof}
\end{lem}
\section{Weinstock inequality in hyperbolic space}\label{sec-4}

From now on, let $R$ be the positive number such that $|B(R)|=|\Omega|$ in $\mathbb{H}^n$.

\subsection{The auxiliary function $h$ and an inequality from the IMCF} $\ $

To simplify the exposition, we define the following integral function $\mathcal{V}(t)$, which represents the volume-related term in our monotone quantity:
\begin{equation}
\mathcal{V}(t):=\int_{B(t)}\frac{\lambda'(r)g(r)}{\lambda(r)}dv = \omega_{n-1}\int_0^t \lambda^{n-2}(s)\lambda'(s)g(s) ds,\qquad t>0.
\end{equation}
Clearly, $\mathcal{V}(t)$ is a strictly increasing function of $t$, and its range is $(0,\infty)$. Thus, we can define an auxiliary function $h:(0,\infty)\to (0,\infty)$ as follows: 
\begin{align}\label{def-h}
h(\mathcal{V}(t))=\frac{1}{|S(t)|}.
\end{align}

The following properties hold for the function $h$.
\begin{lem}
The function $h$ is a decreasing and log-convex function, satisfying the following differential identity:
\begin{align}\label{relationofh'h}
\frac{h'(\mathcal{V}(t))}{h(\mathcal{V}(t))}=-\frac{n-1}{|B(t)|}.
\end{align}
\begin{proof}
Since $\mathcal{V}(t)$ is an increasing function of $t$, while the spherical volume $|S(t)|=\omega_{n-1}\lambda(t)^{n-1}$ is an increasing function of $t$, the composite function $h(\mathcal{V}) = \frac{1}{|S(t(\mathcal{V}))|}$ is obviously decreasing with respect to $\mathcal{V}$. 

Next, taking the logarithm of both sides of \eqref{def-h} yields $\ln h(\mathcal{V}(t)) = -\ln |S(t)|$. Differentiating with respect to $t$ using the chain rule, we have:
\begin{align*}
    \frac{h'(\mathcal{V})}{h(\mathcal{V})} \cdot \mathcal{V}'(t) = -\frac{|S(t)|'}{|S(t)|}.
\end{align*}
By definition, $\mathcal{V}'(t)=|S(t)|\frac{\lambda'(t)g(t)}{\lambda(t)}$. Substituting this into the equation yields
\begin{align*}
    \frac{h'(\mathcal{V})}{h(\mathcal{V})}|S(t)|\frac{\lambda'(t)g(t)}{\lambda(t)}=-(n-1)\frac{\lambda'(t)}{\lambda(t)}.
\end{align*}
Canceling the common factor $\frac{\lambda'(t)}{\lambda(t)}$ and using $g(t)=\frac{|B(t)|}{|S(t)|}$, we obtain \eqref{relationofh'h}.

Finally, from the identity \eqref{relationofh'h}, we see that $(\ln h)'(\mathcal{V}) = -\frac{n-1}{|B(t(\mathcal{V}))|}$. As $\mathcal{V}$ increases, $t$ increases, and $|B(t)|$ increases monotonically. Therefore, $-\frac{n-1}{|B(t(\mathcal{V}))|}$ is increasing. This means $(\ln h)''(\mathcal{V}) > 0$, implying that $h$ is log-convex.
\end{proof}
\end{lem}

Using the properties of $h$, we have the following corollary:
\begin{cor}\label{cor-h-lowerbound}
For any bounded domain $\Omega \subset \mathbb{H}^n$, the auxiliary function $h$ satisfies
\begin{equation}
h'\left(\int_{\Omega}\frac{\lambda'g}{\lambda}dv\right)\geq -\frac{n-1}{|\Omega|}h\left(\int_{\Omega}\frac{\lambda'g}{\lambda}dv\right).
\end{equation}
\begin{proof}
By Proposition \ref{prop-l'g/l}, the function $\frac{\lambda'g}{\lambda}$ is increasing. The mass transplantation Theorem \ref{masstransplantation} yields
\begin{align}\label{eq-int-l'g/l}
\int_{\Omega}\frac{\lambda'g}{\lambda} dv\geq \int_{B(R)}\frac{\lambda'g}{\lambda} dv.
\end{align}
Furthermore, combining this with the log-convexity of $h$ and \eqref{relationofh'h}, we obtain
\begin{align*}
    \frac{h'(\int_{\Omega}\frac{\lambda'g}{\lambda}dv)}{h(\int_{\Omega}\frac{\lambda'g}{\lambda}dv)}\geq \frac{h'(\int_{B(R)}\frac{\lambda'g}{\lambda}dv)}{h(\int_{B(R)}\frac{\lambda'g}{\lambda}dv)}=-\frac{n-1}{|B(R)|}=-\frac{n-1}{|\Omega|}.
\end{align*}
\end{proof}
\end{cor}

Suppose that $\Omega$ is a smooth bounded domain in $\mathbb{H}^n$ with a star-shaped mean convex boundary $\Sigma=\partial\Omega$. Then $\Sigma$ can be evolved along the IMCF as described in Subsection \ref{subsec-imcf}. We now consider the following monotone quantity along the IMCF.

\begin{prop}
Along the IMCF, the quantity
\[
A(t):=|\Sigma_t|^{-1}\frac{\int_{\Sigma_t}g d\mu_t}{|\Omega_t|h(\int_{\Omega_t}\frac{\lambda'g}{\lambda} dv_t)}
\]
is decreasing, i.e., $A'(t)\leq 0$, and $A'(t)=0$ if and only if $\Sigma_t$ is a geodesic sphere centered at the origin.
\begin{proof}
Using the evolution equations \eqref{eq-dmut} and \eqref{eq-dvt}, we take the logarithmic derivative of $A(t)$ to simplify the calculation:
\begin{equation*}
\frac{A'(t)}{A(t)}=-1-\frac{\int_{\Sigma_t}\frac{1}{H}d\mu_t}{|\Omega_t|}+\frac{(\int_{\Sigma_t}g d\mu_t)'}{\int_{\Sigma_t}g d\mu_t}-\frac{(h(\int_{\Omega_t}\frac{\lambda'g}{\lambda}dv))'}{h(\int_{\Omega_t}\frac{\lambda'g}{\lambda}dv)}.
\end{equation*}

Next, we estimate the terms involving the derivatives of integrals. First, for the evolution of the boundary integral $\int_{\Sigma_t}g d\mu_t$, we have
\begin{equation*}
\left(\int_{\Sigma_t}g d\mu_t\right)'=\int_{\Sigma_t}\left(g+\frac{g'\la \partial_r,\nu\ra}{H}\right) d\mu_t \leq \int_{\Sigma_t} \left(g+\frac{g'}{H}\right) d\mu_t,
\end{equation*}
since $g' \ge 0$ by Proposition \ref{prop-g} and $\la \partial_r, \nu \ra \le 1$.

Second, for the evolution of the auxiliary function $h$, applying the chain rule and Corollary \ref{cor-h-lowerbound} yields
\begin{equation*}
h\left(\int_{\Omega_t}\frac{\lambda'g}{\lambda} dv\right)' \geq -\frac{n-1}{|\Omega_t|}h\left(\int_{\Omega_t}\frac{\lambda'g}{\lambda} dv\right)\cdot \int_{\Sigma_t}\frac{\lambda'g}{\lambda H}d\mu_t.
\end{equation*}

Substituting these inequalities back into the logarithmic derivative formula, noticing that the $-1$ cancels with the $g$ term in the numerator expansion, we get:
\begin{align*}
\frac{A'(t)}{A(t)}&\leq \frac{1}{\int_{\Sigma_t}g d\mu_t}\int_{\Sigma_t}\frac{g'}{H}d\mu_t-\frac{1}{|\Omega_t|}\int_{\Sigma_t}\frac{1-(n-1)\frac{\lambda'g}{\lambda}}{H} d\mu_t.
\end{align*}
Using the differential equation $g'=1-(n-1)\frac{\lambda'g}{\lambda}$ (see \eqref{eq-deri-g}), the numerator of the last term is exactly $g'$. Therefore:
\begin{equation*}
\frac{A'(t)}{A(t)}\leq \int_{\Sigma_t}\frac{g'}{H}d\mu_t\left(\frac{1}{\int_{\Sigma_t}gd\mu_t}-\frac{1}{|\Omega_t|}\right).
\end{equation*}
By Corollary \ref{cor-g}, $\int_{\Sigma_t}gd\mu_t\geq |\Omega_t|$. Since $g'\geq 0$ and $H>0$, this implies $A'(t) \le 0$. Equality holds if and only if $\int_{\Sigma_t}g d\mu_t= |\Omega_t|$, which is equivalent to $\Sigma_t$ being a geodesic sphere centered at the origin.
\end{proof}
\end{prop}

As a corollary, using the asymptotic behavior of IMCF and the Cauchy-Schwarz inequality, we have:
\begin{cor}\label{cor-int-g}
\begin{align}\label{eq-intg2-imcf}
\int_{\Sigma}g^2 d\mu\geq |\Sigma||\Omega|^2h\left(\int_{\Omega}\frac{\lambda'g}{\lambda}dv\right)^2.
\end{align}
\begin{proof}
Let $R_t$ be the radius such that $|B(R_t)| = |\Omega_t|$. By the asymptotic behavior of $\Sigma_t$ along IMCF proved by Gerhardt \cite{Ger11}, $\Sigma_t$ expands to infinity. Thus $\lim_{t\to\infty} \min_{\Sigma_t} r = \lim_{t\to\infty} R_t = +\infty$.

By the monotonicity of $h$ and \eqref{eq-int-l'g/l}, we have
\[\begin{aligned}
A(t)=&|\Sigma_t|^{-1}\frac{\int_{\Sigma_t}g d\mu_t}{|\Omega_t|h(\int_{\Omega_t}\frac{\lambda'g}{\lambda} dv_t)}\geq|\Sigma_t|^{-1}\frac{\int_{\Sigma_t}g d\mu_t}{|\Omega_t|h(\int_{B(R_t)}\frac{\lambda'g}{\lambda} dv_t)}\\
=&\frac{|S(R_t)|\int_{\Sigma_t}g d\mu_t}{|\Sigma_t||B(R_t)|}=\frac{1}{g(R_t)}\cdot \frac{\int_{\Sigma_t}gd\mu_t}{|\Sigma_t|}.
\end{aligned}\]
By the Mean Value Theorem, there exists $\xi_t\in [\min_{\Sigma_t}r,\max_{\Sigma_t}r]$ such that $\frac{1}{|\Sigma_t|}\int_{\Sigma_t}gd\mu_t=g(\xi_t)$. Since $A(t)$ is decreasing, taking the limit yields:
\begin{align*}
    \frac{\int_{\Sigma}g d\mu}{|\Omega||\Sigma|h(\int_{\Omega}\frac{\lambda'g}{\lambda} dv)}=A(0)\geq \liminf_{t\to\infty}A(t)\geq \liminf_{t\to\infty}\frac{g(\xi_t)}{g(R_t)}=\frac{n-1}{n-1}=1.
\end{align*}
Combining this with the Cauchy-Schwarz inequality $\int_{\Sigma}g^2 d\mu \geq \frac{1}{|\Sigma|}\left(\int_{\Sigma}g d\mu\right)^2$, the conclusion follows.
\end{proof}
\end{cor}

\subsection{Proof of Theorem \ref{main thm}}\label{subsection4.3}$\ $

The main goal of this section is to prove Theorem \ref{main thm}. To achieve this, we first establish the following crucial inequality for volume-preserving domains.

\begin{thm}\label{thm-core-isoperimetric}
Let $n\geq 3$, $\Omega\subset \mathbb{H}^n$ be a smooth bounded domain with a star-shaped mean convex boundary $\Sigma=\partial\Omega$. If $B(R)$ is a geodesic ball with the same volume as $\Omega$, i.e., $|B(R)|=|\Omega|$, then
\begin{align}\label{eq-prelem}
|\Sigma|\sigma_1(\Omega)\leq |S(R)|\sigma_1(B(R)),
\end{align}
and equality holds if and only if $\Omega$ is a geodesic ball.
\end{thm}

Assuming Theorem \ref{thm-core-isoperimetric} holds, we can readily prove Theorem \ref{main thm}.
\begin{proof}[Proof of Theorem \ref{main thm}]
Let $R^{\ast}$ be the radius of the geodesic ball $\Omega^{\ast}$ that has the same surface area as $\Omega$, i.e., $|S(R^{\ast})|=|\partial\Omega|$. Since $R$ is the radius of the geodesic ball with the same volume as $\Omega$, the classical isoperimetric inequality in hyperbolic space (Theorem \ref{thm-isoperimetric inequality}) implies
\begin{equation}
|S(R)|\leq |\partial\Omega|=|S(R^{\ast})|.
\end{equation}
Since the surface area function $|S(r)|$ is strictly increasing, we deduce $R\leq R^{\ast}$. By Lemma \ref{lem-ln-1g'/g}, the geometric quantity $|S(r)|\sigma_1(B(r))$ is an increasing function of the radius $r$. Therefore,
\begin{equation}
|S(R)|\sigma_1(B(R))\leq |S(R^{\ast})|\sigma_1(B(R^{\ast}))=|\partial\Omega|\sigma_1(\Omega^{\ast}).
\end{equation}
Substituting the conclusion of Theorem \ref{thm-core-isoperimetric} into the left-hand side, we obtain $\sigma_1(\Omega)\leq \sigma_1(\Omega^{\ast})$. The equality holds if and only if $R = R^{\ast}$, which implies $\Omega$ is a geodesic ball.
\end{proof}

We proceed to prove Theorem \ref{thm-core-isoperimetric} by separating it into two cases: higher dimensions ($n \geq 5$) and lower dimensions ($n = 3, 4$).

\subsubsection{The case $n \geq 5$}

When $n \geq 5$, we introduce a specific function and verify its monotonicity:
\begin{lem}\label{lem-n=5}
If $n\geq 5$, then the function
\[
\Phi(r):=(g')^2+(n-1)\frac{g^2}{\lambda^2}+\frac{2(n-1)}{n}\frac{\lambda'g}{\lambda}
\]
is decreasing in $r$. 
\begin{proof}
By taking a derivative and using $g''=(n-1)(\frac{g}{\lambda^2}-\frac{\lambda'g'}{\lambda})$ along with the identity $g'-\frac{\lambda'g}{\lambda}=\frac{n}{n-1}(g'-\frac{1}{n})$, we have
\[\begin{aligned}
\Phi'(r)
=&\left((g')^2+(n-1)\frac{g^2}{\lambda^2}+\frac{2}{n}(1-g')\right)'\\
=&2g''\left(g'-\frac{1}{n}\right)+2(n-1)\frac{g}{\lambda}\frac{g'-\frac{\lambda'g}{\lambda}}{\lambda}\\
=&2(n-1)\frac{g}{\lambda^2}\left(g'-\frac{1}{n}\right)\left(\frac{2n-1}{n-1}-\frac{\lambda\lambda'g'}{g}\right).
\end{aligned}\]
By Proposition \ref{prop-l'g/l}, $g'\leq \frac{1}{n}$, so the term $(g'-\frac{1}{n})$ is non-positive. To prove $\Phi'(r)\leq 0$, it suffices to show that the last term is non-negative, which is equivalent to
\begin{align}\label{eq-deccondi}
\frac{\lambda\lambda'g'}{g}\leq \frac{2n-1}{n-1}.
\end{align}
By Proposition \ref{prop-ll'g'/g}, we know $\frac{\lambda\lambda'g'}{g}\leq \frac{n-1}{n-3}$. Simple algebraic comparison shows that when $n\geq 5$, $\frac{n-1}{n-3} \leq \frac{2n-1}{n-1}$. Thus, the lemma holds.
\end{proof}
\end{lem}

From Lemma \ref{lem-n=5}, by the mass transplantation argument (Theorem \ref{masstransplantation}), we have $\int_{\Omega}\Phi(r)dv\leq \int_{B(R)}\Phi(r)dv$. Expanding and rearranging terms yields:
\begin{align}\label{eq-afterdeccondi}
    \frac{2(n-1)}{n}\left(\int_{\Omega}\frac{\lambda'g}{\lambda}dv-\int_{B(R)}\frac{\lambda'g}{\lambda}dv\right)\leq g'(R)|B(R)|-\int_{\Omega}\left((g')^2+(n-1)\frac{g^2}{\lambda^2}\right)dv.
\end{align}

Using this inequality, we can obtain a fine lower bound for $h^2$:
\begin{lem}\label{lem-h2-bound}
Under the above assumptions, we have:
\begin{equation}
h\left(\int_{\Omega}\frac{\lambda'g}{\lambda}dv\right)^2\geq \frac{1}{g'(R)|B(R)||S(R)|^2}\int_{\Omega}\left((g')^2+(n-1)\frac{g^2}{\lambda^2}\right)dv.
\end{equation}
\begin{proof}
By the Mean Value Theorem, there exists $\xi\in [\int_{B(R)}\frac{\lambda'g}{\lambda}dv,\int_{\Omega}\frac{\lambda'g}{\lambda}dv]$ such that
\[\begin{aligned}
h\left(\int_{\Omega}\frac{\lambda'g}{\lambda}dv\right)^2-h\left(\int_{B(R)}\frac{\lambda'g}{\lambda}dv\right)^2=2h(\xi)h'(\xi)\left(\int_{\Omega}\frac{\lambda'g}{\lambda}dv-\int_{B(R)}\frac{\lambda'g}{\lambda}dv\right).
\end{aligned}\]
Recall that the function $h(\int_{B(r)}\frac{\lambda'g}{\lambda}dv)h'(\int_{B(r)}\frac{\lambda'g}{\lambda}dv)=-\frac{n-1}{|B(r)||S(r)|^2}$ is increasing in $r$. Since the equivalent radius for $\xi$ is at least $R$, we have $h(\xi)h'(\xi) \geq -\frac{n-1}{|B(R)||S(R)|^2}$. 

Combining this with $g'(R)\leq \frac{1}{n}$ and \eqref{eq-afterdeccondi}, we derive
\begin{align*}
    &h\left(\int_{\Omega}\frac{\lambda'g}{\lambda}dv\right)^2\\
    \geq & \frac{1}{|S(R)|^2}-\frac{2(n-1)}{|B(R)||S(R)|^2}\left(\int_{\Omega}\frac{\lambda'g}{\lambda}dv-\int_{B(R)}\frac{\lambda'g}{\lambda}dv\right)\\
    \geq & \frac{1}{|S(R)|^2}-\frac{1}{g'(R)|B(R)||S(R)|^2}\frac{2(n-1)}{n}\left(\int_{\Omega}\frac{\lambda'g}{\lambda}dv-\int_{B(R)}\frac{\lambda'g}{\lambda}dv\right)\\
    \geq & \frac{1}{|S(R)|^2}-\frac{1}{g'(R)|B(R)||S(R)|^2}\left(g'(R)|B(R)|-\int_{\Omega}\left((g')^2+(n-1)\frac{g^2}{\lambda^2}\right)dv\right)\\
    = & \frac{1}{g'(R)|B(R)||S(R)|^2}\int_{\Omega}\left((g')^2+(n-1)\frac{g^2}{\lambda^2}\right)dv.
\end{align*}
\end{proof}
\end{lem}

\begin{proof}[Proof of Theorem \ref{thm-core-isoperimetric} for $n\geq 5$]
According to Proposition \ref{prop-upperbound-steklov}, we have the variational upper bound $\sigma_1(\Omega) \leq \frac{\int_{\Omega}((g')^2+(n-1)\frac{g^2}{\lambda^2})dv}{\int_{\partial\Omega}g^2d\mu}$. Substituting the boundary integral lower bound from Corollary \ref{cor-int-g} into the denominator, we get
\begin{equation}
|\partial\Omega|\sigma_1(\Omega) \leq \frac{\int_{\Omega}\left((g')^2+(n-1)\frac{g^2}{\lambda^2}\right)dv}{|\Omega|^2 h\left(\int_{\Omega}\frac{\lambda'g}{\lambda}dv\right)^2}.
\end{equation}
By substituting the lower bound for $h^2$ from Lemma \ref{lem-h2-bound}, the integral terms cancel out:
\begin{equation}
|\partial\Omega|\sigma_1(\Omega) \leq \frac{g'(R)|B(R)||S(R)|^2}{|\Omega|^2}.
\end{equation}
Using the equal volume condition $|\Omega|=|B(R)|$ and $\sigma_1(B(R))=\frac{g'(R)}{g(R)}$, this simplifies to $|S(R)|\sigma_1(B(R))$.
\end{proof}

\subsubsection{The case $n = 3, 4$}

In lower dimensions ($n = 3, 4$), the function $\Phi(r)$ from Lemma \ref{lem-n=5} is no longer globally monotonic. To overcome this difficulty, we introduce a Weighted Mass Transplantation method.

In lower dimensions ($n = 3, 4$), the function $\Phi(r)$ from Lemma \ref{lem-n=5} is no longer globally monotonic. To overcome this difficulty, we introduce a weighted variant of the classical mass transplantation method. Our weighted version shares the same spirit as the proof of the Bathtub principle (see, e.g., Lieb and Loss \cite[Theorem 1.14]{LL01}).

\begin{thm}[Weighted Mass Transplantation]\label{thm-weighted-mt}
Let $\Omega \subset \mathbb{H}^n$ be a bounded measurable set, and $B(R)$ be a geodesic ball centered at the origin. If $f:[0,\infty)\to \mathbb{R}$ is a decreasing integrable function, and $F(x)$ is an arbitrary non-negative integrable weight function, then we have:
\begin{equation}
\int_{\Omega}f(r(x))F(x)dv-\int_{B(R)}f(r(x))F(x)dv\leq f(R)\left(\int_{\Omega}F(x)dv-\int_{B(R)}F(x)dv\right).
\end{equation}
where $r(x)$ denotes the geodesic distance from $x$ to the origin. If $f$ is strictly decreasing and $F(x)>0$ almost everywhere, equality holds if and only if $\Omega=B(R)$.
\begin{proof}
We introduce the auxiliary function 
\[
J(x) := F(x)\left(f(r(x))-f(R)\right). 
\]
The desired inequality is equivalent to $\int_{\Omega}J(x)dv-\int_{B(R)}J(x)dv\leq 0$.
By the additivity of integrals, the difference can be decomposed as:
\begin{equation*}
\int_{\Omega}J(x)dv - \int_{B(R)}J(x)dv = \int_{\Omega\setminus B(R)}J(x)dv - \int_{B(R)\setminus \Omega}J(x)dv.
\end{equation*}
In the exterior region $x \in \Omega\setminus B(R)$, $r(x) \geq R$. Since $f$ is decreasing, $f(r(x)) \leq f(R)$. Because $F(x) \geq 0$, we have $J(x) \leq 0$.

Conversely, in the interior region $x \in B(R)\setminus \Omega$, $r(x) \leq R$, yielding $f(r(x)) \geq f(R)$ and thus $J(x) \geq 0$.

Combining these facts, the integral difference satisfies 
\[
\int_{\Omega\setminus B(R)}Jdv - \int_{B(R)\setminus \Omega}Jdv \leq 0 - 0 = 0.
\]
\end{proof}
\end{thm}

Next, we establish a new monotonicity lemma specifically tailored for lower dimensions:
\begin{lem}\label{lem-n=3-4}
If $n=3, 4$, the function
\begin{equation}
\Psi(r):=g'(r)+(n-1)\frac{g^2(r)}{\lambda^2(r)g'(r)}+2(n-1)\frac{\lambda'(r)g(r)}{\lambda(r)}
\end{equation}
is strictly decreasing in $r$.
\begin{proof}
We write $\Psi(r) = 1 + (n-1)\left(\frac{g^2(r)}{\lambda^2(r)g'(r)} + \frac{\lambda'(r)g(r)}{\lambda(r)}\right)$ and analyze the dimensions separately.

\textbf{Case $n=3$:} We compute 
\[
g(r) = \frac{\lambda(r)\lambda'(r) - r}{2\lambda^2(r)}. 
\]
Differentiating and using $\lambda''(r)=\lambda(r)$ gives 
\[
g'(r) = \frac{r\lambda'(r) - \lambda(r)}{\lambda^3(r)}. 
\]
Substituting these into $\Psi(r)$ and bringing them over a common denominator $2\lambda^3(r\lambda' - \lambda)$, we obtain the rational expression:
\begin{equation}
\Psi(r) = \frac{4r\lambda\lambda'^3 - 3\lambda^2\lambda'^2 - 2r^2\lambda'^2 - 2r\lambda\lambda' + 2\lambda^2 + r^2}{2\lambda^3(r\lambda' - \lambda)} := \frac{U(r)}{V(r)}.
\end{equation}
Taking the derivative using the quotient rule, 
\[
\Psi'(r) = \frac{U'(r)V(r) - U(r)V'(r)}{V^2(r)}. 
\]
We repeatedly apply $\lambda''(r) = \lambda(r)$ and $\lambda'^2(r) = \lambda^2(r) + 1$ to eliminate second derivatives and reduce the degree of $\lambda'$, we obtain 
\[
U'V - UV' = 2\lambda^2(r)M(r),
\]
where
\begin{equation*}
M(r) = -\lambda^5\lambda' + (11r+4r^3)\lambda^4 - (3+10r^2)\lambda^3\lambda' + (9r+6r^3)\lambda^2 - 9r^2\lambda\lambda' + 3r^3.
\end{equation*}
Let $N(r) = \frac{M(r)}{\lambda^6(r)}$. Differentiating $N(r)$ and simplifying yields:
\begin{equation*}
N'(r) = -\frac{2}{\lambda^7(r)} \big(r\lambda'(r) - \lambda(r)\big) \Big( r\big(2 + \lambda^2(r) + \lambda'^2(r)\big) - 3\lambda(r)\lambda'(r) \Big)^2.
\end{equation*}
Since $r\lambda'(r) - \lambda(r) > 0$ for $r>0$, we have $N'(r) \le 0$, meaning $N(r)$ is strictly decreasing on $(0, \infty)$. 

To find the initial value at the origin, we expand $M(r)$ using Taylor series up to $O(r^7)$. The lower-order terms identically cancel out (the coefficients of $r^3$ and $r^5$ sum to zero), so $M(r) = O(r^7)$. Thus $\lim_{r \to 0^+} N(r) = 0$. 

Since $N(r)$ starts from $0$ and is strictly decreasing, $N(r) < 0$ for all $r>0$. Because 
\[
\Psi'(r) = \frac{\lambda^2(r) N(r)}{2\big(r\lambda'(r) - \lambda(r)\big)^2}, 
\]
we conclude that $\Psi'(r) < 0$. 

\textbf{Case $n=4$:} Here $g(r) = \frac{(\lambda'(r)-1)^2(\lambda'(r)+2)}{3\lambda^3(r)}$ and $g'(r) = \frac{1}{(\lambda'(r)+1)^2}$. Thus,
\begin{equation*}
\Psi(r) = \frac{5}{2} - \frac{1}{6}\left(1 - \frac{2}{\lambda'(r)+1}\right)^2.
\end{equation*}
Since $\lambda'(r) \ge 1$ and is strictly increasing, the squared term $\left( 1 - \frac{2}{\lambda'(r)+1} \right)^2$ is strictly increasing. Combined with the negative sign, $\Psi(r)$ is strictly decreasing.
\end{proof}
\end{lem}

Now we apply the weighted mass transplantation theorem \ref{thm-weighted-mt}. We choose the decreasing function $f(r)=g'(r)$ and the non-negative weight function $F(r)=g'(r)+(n-1)\frac{g^2(r)}{\lambda^2(r)g'(r)}$. Notice that $F(r)=\Psi(r)-2(n-1)\frac{\lambda'(r)g(r)}{\lambda(r)}$.
Applying Theorem \ref{thm-weighted-mt} yields:
\begin{equation}\label{eq-n=3-4.7''}
\begin{aligned}
&\int_{\Omega}\left((g')^2+(n-1)\frac{g^2}{\lambda^2}\right)dv-\int_{B(R)}\left((g')^2+(n-1)\frac{g^2}{\lambda^2}\right)dv\\
=&\int_{\Omega}g'(r)F(r)dv-\int_{B(R)}g'(r)F(r)dv\\
\leq & g'(R)\left(\int_{\Omega}\Psi dv-\int_{B(R)}\Psi dv\right)-2(n-1)g'(R)\left(\int_{\Omega}\frac{\lambda'g}{\lambda}dv-\int_{B(R)}\frac{\lambda'g}{\lambda}dv\right)\\
\leq & -2(n-1)g'(R)\left(\int_{\Omega}\frac{\lambda'g}{\lambda}dv-\int_{B(R)}\frac{\lambda'g}{\lambda}dv\right),
\end{aligned}
\end{equation}
where the last step uses the fact that $\int_{\Omega}\Psi dv \leq \int_{B(R)}\Psi dv$ due to the standard mass transplantation on the decreasing function $\Psi(r)$.

\begin{proof}[Proof of Theorem \ref{thm-core-isoperimetric} for $n=3,4$]
Recall the lower bound for $h^2$ established in Lemma \ref{lem-h2-bound} (which relies only on the Mean Value Theorem and the properties of $h$)
\begin{equation*}
h\left(\int_{\Omega}\frac{\lambda'g}{\lambda}dv\right)^2 \geq \frac{1}{|S(R)|^2}-\frac{2(n-1)}{|B(R)||S(R)|^2}\left(\int_{\Omega}\frac{\lambda'g}{\lambda}dv-\int_{B(R)}\frac{\lambda'g}{\lambda}dv\right).
\end{equation*}
Substituting inequality \eqref{eq-n=3-4.7''} into this lower bound, and using the identity on the ball $\int_{B(R)}((g')^2+(n-1)\frac{g^2}{\lambda^2})dv=g'(R)|B(R)|$, we obtain
\begin{equation*}\begin{aligned}
&h\left(\int_{\Omega}\frac{\lambda'g}{\lambda}dv\right)^2\\
\geq& \frac{1}{|S(R)|^2}-\frac{1}{g'(R)|B(R)||S(R)|^2}\left(g'(R)|B(R)|-\int_{\Omega}\left((g')^2+(n-1)\frac{g^2}{\lambda^2}\right)dv\right)\\
= & \frac{1}{g'(R)|B(R)||S(R)|^2}\int_{\Omega}\left((g')^2+(n-1)\frac{g^2}{\lambda^2}\right)dv.
\end{aligned}\end{equation*}
This establishes the exact same lower bound for $h^2$ as Lemma \ref{lem-h2-bound} does for $n \ge 5$. The subsequent steps bounding $\sigma_1$ are identical to the $n \ge 5$ case.
\end{proof}

\subsection{Proof of Corollary \ref{main cor}}$\ $

Let $B(R)$ be a geodesic ball with the radial part of the first non-zero Steklov eigenfunction denoted as $g(r)$. Following the construction method in \cite{Ver21}, by using the Borsuk-Ulam theorem, we can choose an appropriate origin $p \in \Omega$ and rotate the normal coordinate system $(x_1, \dots, x_n)$ such that the test functions satisfy the orthogonality conditions against the first $n-1$ non-zero eigenfunctions $u_1, u_2, \dots, u_{n-1}$ on the boundary:
\begin{equation}
\int_{\partial\Omega} g(r)\frac{x_i}{r} u_j d\mu = 0, \qquad 0 \le j \le i-1, \quad 1 \le i \le n.
\end{equation}

According to the variational characterization \eqref{def-steklov}, substituting the test functions $g(r)\frac{x_i}{r}$ into the Rayleigh quotient gives
\begin{equation}
\sigma_i(\Omega)\int_{\partial\Omega}\left(g(r)\frac{x_i}{r}\right)^2 d\mu \le \int_{\Omega}\left\| \bar\nabla \left(g(r)\frac{x_i}{r}\right) \right\|^2 dv.
\end{equation}
Dividing both sides by $\sigma_i(\Omega)$, summing from $i=1$ to $n$, and using the identity $\sum_{i=1}^n \frac{x_i^2}{r^2} = 1$, we get
\begin{equation}\label{eq-weinstock-sum1}
\int_{\partial\Omega}g^2(r)d\mu \le \sum_{i=1}^n \frac{1}{\sigma_i(\Omega)}\int_{\Omega}\frac{g^2(r)}{\lambda^2(r)}dv + \sum_{i=1}^n \frac{1}{\sigma_i(\Omega)}\int_{\Omega}\left( (g'(r))^2 - \frac{g^2(r)}{\lambda^2(r)} \right)\frac{x_i^2}{r^2}dv.
\end{equation}

We can reorganize the second term on the right side of \eqref{eq-weinstock-sum1} to isolate the first $n-1$ terms
\begin{equation}
\sum_{i=1}^n \frac{1}{\sigma_i(\Omega)}\frac{x_i^2}{r^2} = \sum_{i=1}^{n-1}\left( \frac{1}{\sigma_i(\Omega)} - \frac{1}{\sigma_n(\Omega)} \right)\frac{x_i^2}{r^2} + \frac{1}{\sigma_n(\Omega)}.
\end{equation}
Due to the properties of radial functions in hyperbolic space, $(g'(r))^2 - \frac{g^2(r)}{\lambda^2(r)} \le 0$. Since $\frac{1}{\sigma_i(\Omega)} - \frac{1}{\sigma_n(\Omega)} \ge 0$ for $1 \le i \le n-1$, it implies
\begin{equation}
\sum_{i=1}^{n-1}\left(\frac{1}{\sigma_i(\Omega)} - \frac{1}{\sigma_n(\Omega)}\right)\int_{\Omega}\left( (g'(r))^2 - \frac{g^2(r)}{\lambda^2(r)} \right)\frac{x_i^2}{r^2}dv \le 0.
\end{equation}

Substituting this back into \eqref{eq-weinstock-sum1}, only the part with $\frac{1}{\sigma_n(\Omega)}$ remains in the second integral term, leading to cancellation:
\begin{equation}\label{eq-weinstock-cancel}
\begin{aligned}
&\int_{\partial\Omega}g^2(r)d\mu \\
\le& \left(\sum_{i=1}^{n-1}\frac{1}{\sigma_i(\Omega)} + \frac{1}{\sigma_n(\Omega)}\right) \int_{\Omega} \frac{g^2(r)}{\lambda^2(r)} dv + \frac{1}{\sigma_n(\Omega)} \int_{\Omega} \left( (g'(r))^2 - \frac{g^2(r)}{\lambda^2(r)} \right) dv \\
=& \sum_{i=1}^{n-1}\frac{1}{\sigma_i(\Omega)} \int_{\Omega} \frac{g^2(r)}{\lambda^2(r)} dv + \frac{1}{\sigma_n(\Omega)} \int_{\Omega} (g'(r))^2 dv.
\end{aligned}
\end{equation}
By the ordering of eigenvalues, $\frac{1}{\sigma_n(\Omega)} \le \frac{1}{n-1}\sum_{i=1}^{n-1}\frac{1}{\sigma_i(\Omega)}$. Substituting this into \eqref{eq-weinstock-cancel} and factoring out $\frac{1}{n-1}\sum_{i=1}^{n-1}\frac{1}{\sigma_i(\Omega)}$, we obtain
\begin{equation}\label{eq-weinstock-sum2}
\int_{\partial\Omega}g^2(r)d\mu \le \frac{1}{n-1}\sum_{i=1}^{n-1}\frac{1}{\sigma_i(\Omega)} \int_{\Omega} \left( (n-1)\frac{g^2(r)}{\lambda^2(r)} + (g'(r))^2 \right) dv.
\end{equation}

Since $\Omega$ is star-shaped and mean convex, and $\Omega^{\ast}$ is a geodesic ball with the same surface area, our previously established estimate (combining Corollary \ref{cor-int-g} and Lemma \ref{lem-h2-bound}) yields
\begin{equation}
\int_{\partial\Omega}g^2(r)d\mu\geq \frac{1}{\sigma_1(\Omega^{\ast})}\int_{\Omega}\left( (n-1)\frac{g^2(r)}{\lambda^2(r)} + (g'(r))^2 \right) dv.
\end{equation}
Substituting this lower bound into \eqref{eq-weinstock-sum2}, we get
\begin{equation}
\frac{1}{\sigma_1(\Omega^{\ast})}\leq \frac{1}{n-1}\sum_{i=1}^{n-1}\frac{1}{\sigma_i(\Omega)}.
\end{equation}
Note that on the geodesic ball $\Omega^{\ast}$, the first $n-1$ non-zero Steklov eigenvalues are repeated roots, meaning $\sigma_1(\Omega^{\ast}) = \sigma_i(\Omega^{\ast})$ for $1 \le i \le n-1$. Substituting this yields:
\begin{equation}
\sum_{i=1}^{n-1}\frac{1}{\sigma_i(\Omega^{\ast})} \le \sum_{i=1}^{n-1}\frac{1}{\sigma_i(\Omega)}.
\end{equation}
The condition for equality is identical to the one established in our previous integral estimates, which holds if and only if $\Omega$ is a geodesic ball. This completes the proof of Corollary \ref{main cor}. \qed


\begin{bibdiv}
\begin{biblist}

\bib{Ahl50}{article}{
   author={Ahlfors, Lars V.},
   title={Open Riemann surfaces and extremal problems on compact subregions},
   journal={Comment. Math. Helv.},
   volume={24},
   date={1950},
   pages={100--134},
}

\bib{AS96}{article}{
   author={Aithal, A. R.},
   author={Santhanam, G.},
   title={Sharp upper bound for the first non-zero Neumann eigenvalue for bounded domains in rank-$1$ symmetric spaces},
   journal={Trans. Amer. Math. Soc.},
   volume={348},
   date={1996},
   number={10},
   pages={3955--3965},
}

\bib{AV22}{article}{
   author={Anoop, T. V.},
   author={Verma, Sheela},
   title={Szeg\"{o}-Weinberger type inequalities for symmetric domains in simply connected space forms},
   journal={J. Math. Anal. Appl.},
   volume={515},
   date={2022},
   number={2},
   pages={Paper No. 126429, 16},
}

\bib{Ber05}{article}{
   author={Bernstein, Felix},
   title={\"{U}ber die isoperimetrische Eigenschaft des Kreises auf der Kugeloberfl\"{a}che und in der Ebene},
   journal={Math. Ann.},
   volume={60},
   date={1905},
   number={1},
   pages={117--136},
}

\bib{BS14}{article}{
   author={Binoy},
   author={Santhanam, G.},
   title={Sharp upperbound and a comparison theorem for the first nonzero Steklov eigenvalue},
   journal={J. Ramanujan Math. Soc.},
   volume={29},
   date={2014},
   number={2},
   pages={133--154},
}

\bib{BDR12}{article}{
   author={Brasco, Lorenzo},
   author={De Philippis, Guido},
   author={Ruffini, Berardo},
   title={Spectral optimization for the Stekloff-Laplacian: the stability issue},
   journal={J. Funct. Anal.},
   volume={262},
   date={2012},
   number={11},
   pages={4675--4710},
}

\bib{Bre13}{article}{
   author={Brendle, Simon},
   title={Constant mean curvature surfaces in warped product manifolds},
   journal={Publ. Math. Inst. Hautes \'{E}tudes Sci.},
   volume={117},
   date={2013},
   pages={247--269},
}

\bib{Bro01}{article}{
   author={Brock, F.},
   title={An isoperimetric inequality for eigenvalues of the Stekloff problem},
   journal={ZAMM Z. Angew. Math. Mech.},
   volume={81},
   date={2001},
   number={1},
   pages={69--71},
}

\bib{BFNT21}{article}{
   author={Bucur, Dorin},
   author={Ferone, Vincenzo},
   author={Nitsch, Carlo},
   author={Trombetti, Cristina},
   title={Weinstock inequality in higher dimensions},
   journal={J. Differential Geom.},
   volume={118},
   date={2021},
   number={1},
   pages={1--21},
}

\bib{CR19}{article}{
   author={Castillon, Philippe},
   author={Ruffini, Berardo},
   title={A spectral characterization of geodesic balls in non-compact rank one symmetric spaces},
   journal={Ann. Sc. Norm. Super. Pisa Cl. Sci. (5)},
   volume={19},
   date={2019},
   number={4},
   pages={1359--1388},
}

\bib{CEG11}{article}{
   author={Colbois, Bruno},
   author={El Soufi, Ahmad},
   author={Girouard, Alexandre},
   title={Isoperimetric control of the Steklov spectrum},
   journal={J. Funct. Anal.},
   volume={261},
   date={2011},
   number={5},
   pages={1384--1399},
}

\bib{CEG19}{article}{
   author={Colbois, Bruno},
   author={El Soufi, Ahmad},
   author={Girouard, Alexandre},
   title={Compact manifolds with fixed boundary and large Steklov eigenvalues},
   journal={Proc. Amer. Math. Soc.},
   volume={147},
   date={2019},
   number={9},
   pages={3813--3827},
}

\bib{CGGS24}{article}{
   author={Colbois, Bruno},
   author={Girouard, Alexandre},
   author={Gordon, Carolyn},
   author={Sher, David},
   title={Some recent developments on the Steklov eigenvalue problem},
   journal={Rev. Mat. Complut.},
   volume={37},
   date={2024},
   number={1},
   pages={1--161},
}

\bib{Esc99}{article}{
   author={Escobar, Jos\'{e} F.},
   title={An isoperimetric inequality and the first Steklov eigenvalue},
   journal={J. Funct. Anal.},
   volume={165},
   date={1999},
   number={1},
   pages={101--116},
}

\bib{FS11}{article}{
   author={Fraser, Ailana},
   author={Schoen, Richard},
   title={The first Steklov eigenvalue, conformal geometry, and minimal surfaces},
   journal={Adv. Math.},
   volume={226},
   date={2011},
   number={5},
   pages={4011--4030},
}

\bib{FS16}{article}{
   author={Fraser, Ailana},
   author={Schoen, Richard},
   title={Sharp eigenvalue bounds and minimal surfaces in the ball},
   journal={Invent. Math.},
   volume={203},
   date={2016},
   number={3},
   pages={823--890},
}

\bib{FS19}{article}{
   author={Fraser, Ailana},
   author={Schoen, Richard},
   title={Shape optimization for the Steklov problem in higher dimensions},
   journal={Adv. Math.},
   volume={348},
   date={2019},
   pages={146--162},
}

\bib{FS20}{article}{
   author={Fraser, Ailana},
   author={Schoen, Richard},
   title={Some results on higher eigenvalue optimization},
   journal={Calc. Var. Partial Differential Equations},
   volume={59},
   date={2020},
   number={5},
   pages={Paper No. 151, 22},
}

\bib{FL21}{article}{
   author={Freitas, Pedro},
   author={Laugesen, Richard S.},
   title={From Neumann to Steklov and beyond, via Robin: the Weinberger way},
   journal={Amer. J. Math.},
   volume={143},
   date={2021},
   number={3},
   pages={969--994},
}

\bib{Gab06}{article}{
   author={Gabard, Alexandre},
   title={Sur la repr\'{e}sentation conforme des surfaces de Riemann \`{a} bord et une caract\'{e}risation des courbes s\'{e}parantes},
   journal={Comment. Math. Helv.},
   volume={81},
   date={2006},
   number={4},
   pages={945--964},
}

\bib{Ger11}{article}{
   author={Gerhardt, Claus},
   title={Inverse curvature flows in hyperbolic space},
   journal={J. Differential Geom.},
   volume={89},
   date={2011},
   number={3},
   pages={487--527},
}

\bib{GKL21}{article}{
   author={Girouard, Alexandre},
   author={Karpukhin, Mikhail},
   author={Lagac\'{e}, Jean},
   title={Continuity of eigenvalues and shape optimisation for Laplace and Steklov problems},
   journal={Geom. Funct. Anal.},
   volume={31},
   date={2021},
   number={3},
   pages={513--561},
}

\bib{GLW25}{article}{
   author={Gu, Pingxin},
   author={Li, Haizhong},
   author={Wan, Yao},
   title={Weinstock inequality in hyperbolic space},
   journal={J. Funct. Anal.},
   volume={289},
   date={2025},
   pages={111155},
}

\bib{GP10}{article}{
   author={Girouard, A.},
   author={Polterovich, I.},
   title={On the Hersch-Payne-Schiffer estimates for the eigenvalues of the Steklov problem},
   journal={Funktsional. Anal. i Prilozhen.},
   volume={44},
   date={2010},
   number={2},
   pages={33--47},
}

\bib{GP12}{article}{
   author={Girouard, Alexandre},
   author={Polterovich, Iosif},
   title={Upper bounds for Steklov eigenvalues on surfaces},
   journal={Electron. Res. Announc. Math. Sci.},
   volume={19},
   date={2012},
   pages={77--85},
}

\bib{GP14}{article}{
   author={Girouard, Alexandre},
   author={Parnovski, Leonid},
   author={Polterovich, Iosif},
   author={Sher, David A.},
   title={The Steklov spectrum of surfaces: asymptotics and invariants},
   journal={Math. Proc. Cambridge Philos. Soc.},
   volume={157},
   date={2014},
   number={3},
   pages={379--389},
}

\bib{GP17}{article}{
   author={Girouard, Alexandre},
   author={Polterovich, Iosif},
   title={Spectral geometry of the Steklov problem (survey article)},
   journal={J. Spectr. Theory},
   volume={7},
   date={2017},
   number={2},
   pages={321--359},
}

\bib{Has11}{article}{
   author={Hassannezhad, Asma},
   title={Conformal upper bounds for the eigenvalues of the Laplacian and Steklov problem},
   journal={J. Funct. Anal.},
   volume={261},
   date={2011},
   number={12},
   pages={3419--3436},
}

\bib{HP68}{article}{
   author={Hersch, Joseph},
   author={Payne, Lawrence E.},
   title={Extremal principles and isoperimetric inequalities for some mixed problems of Stekloff's type},
   journal={Z. Angew. Math. Phys.},
   volume={19},
   date={1968},
   pages={802--817},
}

\bib{HPS75}{article}{
   author={Hersch, J.},
   author={Payne, L. E.},
   author={Schiffer, M. M.},
   title={Some inequalities for Stekloff eigenvalues},
   journal={Arch. Rational Mech. Anal.},
   volume={57},
   date={1975},
   pages={99--114},
}

\bib{HL01}{article}{
   author={Hislop, P. D.},
   author={Lutzer, C. V.},
   title={Spectral asymptotics of the Dirichlet-to-Neumann map on multiply connected domains in $\mathbb{R}^d$},
   journal={Inverse Problems},
   volume={17},
   date={2001},
   number={6},
   pages={1717--1741},
}

\bib{Hon21}{article}{
   author={Hong, Han},
   title={Higher dimensional surgery and Steklov eigenvalues},
   journal={J. Geom. Anal.},
   volume={31},
   date={2021},
   number={12},
   pages={11931--11951},
}

\bib{Kar17}{article}{
   author={Karpukhin, Mikhail},
   title={Bounds between Laplace and Steklov eigenvalues on nonnegatively curved manifolds},
   journal={Electron. Res. Announc. Math. Sci.},
   volume={24},
   date={2017},
   pages={100--109},
}

\bib{Kok14}{article}{
   author={Kokarev, Gerasim},
   title={Variational aspects of Laplace eigenvalues on Riemannian surfaces},
   journal={Adv. Math.},
   volume={258},
   date={2014},
   pages={191--239},
}


\bib{KW23}{article}{
   author={Kwong, Kwok-Kun},
   author={Wei, Yong},
   title={Geometric inequalities involving three quantities in warped product manifolds},
   journal={Adv. Math.},
   volume={430},
   date={2023},
   pages={Paper No. 109213, 28},
}

\bib{LL01}{book}{
   author={Lieb, Elliott H.},
   author={Loss, Michael},
   title={Analysis},
   series={Graduate Studies in Mathematics},
   volume={14},
   edition={2},
   publisher={American Mathematical Society, Providence, RI},
   date={2001},
   pages={xx2+346},
}

\bib{LU89}{article}{
   author={Lee, John M.},
   author={Uhlmann, Gunther},
   title={Determining anisotropic real-analytic conductivities by boundary measurements},
   journal={Comm. Pure Appl. Math.},
   volume={42},
   date={1989},
   number={8},
   pages={1097--1112},
}

\bib{Oss78}{article}{
   author={Osserman, Robert},
   title={The isoperimetric inequality},
   journal={Bull. Amer. Math. Soc.},
   volume={84},
   date={1978},
   number={6},
   pages={1182--1238},
}

\bib{Sch39}{article}{
   author={Schmidt, Erhard},
   title={\"{U}ber das isoperimetrische Problem im Raum von $n$ Dimensionen},
   journal={Math. Z.},
   volume={44},
   date={1939},
   number={1},
   pages={689--788},
}

\bib{Sch43}{article}{
   author={Schmidt, Erhard},
   title={Beweis der isoperimetrischen Eigenschaft der Kugel im hyperbolischen und sph\"{a}rischen Raum jeder Dimensionenzahl},
   journal={Math. Z.},
   volume={49},
   date={1943},
   pages={1--109},
}

\bib{Ste02}{article}{
   author={Stekloff, W.},
   title={Sur les probl\`{e}mes fondamentaux de la physique math\'{e}matique (suite et fin)},
   journal={Ann. Sci. \'{E}cole Norm. Sup. (3)},
   volume={19},
   date={1902},
   pages={455--490},
}

\bib{Ver21}{article}{
   author={Verma, Sheela},
   title={An isoperimetric inequality for the harmonic mean of the Steklov eigenvalues in hyperbolic space},
   journal={Arch. Math. (Basel)},
   volume={116},
   date={2021},
   number={2},
   pages={193--201},
}

\bib{Wei54}{article}{
   author={Weinstock, Robert},
   title={Inequalities for a classical eigenvalue problem},
   journal={J. Rational Mech. Anal.},
   volume={3},
   date={1954},
   pages={745--753},
}

\bib{Wei56}{article}{
   author={Weinberger, H. F.},
   title={An isoperimetric inequality for the $N$-dimensional free membrane problem},
   journal={J. Rational Mech. Anal.},
   volume={5},
   date={1956},
   pages={633--636},
}

\bib{YY17}{article}{
   author={Yang, Liangwei},
   author={Yu, Chengjie},
   title={A higher dimensional generalization of Hersch-Payne-Schiffer inequality for Steklov eigenvalues},
   journal={J. Funct. Anal.},
   volume={272},
   date={2017},
   number={10},
   pages={4122--4130},
}

\end{biblist}
\end{bibdiv}
\end{document}